\documentclass[a4paper,11pt]{article}
\usepackage{babel}[english]
\usepackage[margin=3cm,nohead]{geometry}
\usepackage[version=3]{mhchem}
\usepackage{url}
\usepackage{hyperref}
\usepackage{amssymb}
\usepackage{amsthm}
\usepackage{parskip}
\usepackage{footmisc}

\newcommand{\mat}{\text{Mat}}
\renewcommand{\c}{\mathbb{C}}
\newcommand{\q}{\mathbb{Q}}
\newcommand{\z}{\mathbb{Z}}

\makeatletter
\def\blfootnote{\gdef\@thefnmark{}\@footnotetext}
\makeatother

\makeindex

\title{An elementary proof of the Benjamini-Nekrashevych-Pete conjecture for the semi-direct products $\z^n\rtimes \z$}
\author{Dean Wardell}
\date{}

\begin{document}


\maketitle

\theoremstyle{definition}
\newtheorem{definition}{Definition}[section]
\newtheorem*{acknowledgement}{Acknowledgement}

\theoremstyle{plain}
\newtheorem{theorem}[definition]{Theorem}
\newtheorem{claim}[definition]{Claim}
\newtheorem*{claim*}{Claim}
\newtheorem{corollary}[definition]{Corollary}
\newtheorem{question}[definition]{Question}
\newtheorem{proposition}[definition]{Proposition}
\newtheorem{lemma}[definition]{Lemma}
\newtheorem{conjecture}{Conjecture}

\theoremstyle{remark}
\newtheorem{example}[definition]{Example}
\newtheorem{remark}[definition]{Remark}
\newtheorem*{notation}{Notation}
\newtheorem*{term}{Terminology}

\begin{abstract}
    A finitely generated group $G$ is called strongly scale-invariant if there exists an injective homomorphism $f:G\to G$ such that $f(G)$ is a finite index subgroup of $G$ and such that $\cap_{n\geq 0} f^n(G)$ is finite. Nekrashevych and Pete conjectured that all strongly scale-invariant groups are virtually nilpotent, after disproving a stronger conjecture by Benjamini.
    
    This conjecture is known to be true in some situations. Deré proved it for virtually polycyclic groups. In this paper, we provide an elementary proof for those polycyclic groups that can be written as a semi-direct product $\z^n\rtimes \z$.
\end{abstract}

\section{Introduction}

Let $G$ be a finitely generated group. We call $G$ \textbf{scale-invariant} if there exists a descending chain $\{G_k\}_{k\geq0}$ of finite index subgroups in $G$, such that each $G_k$ is isomorphic to $G$ and such that $\cap_{k\geq0} G_k$ is finite. Benjamini introduced this concept, motivated by problems of renormalization in percolation theory on Cayley graphs of groups \cite[Section 9.2]{Sapir}.

Benjamini conjectured that every finitely generated scale-invariant group has polynomial growth. Therefore, by Gromov's theorem \cite{Gromov} every such group is virtually nilpotent. This conjecture was shown to be false by Nekrashevych and Pete by exhibiting a family of counterexamples in \cite{NekrashevychPete}. As a result, Nekrashevych and Pete thought of a (natural) stronger property, which they call strong scale-invariance.

\begin{definition}
    Let $G$ be a group and $f:G\to G$ a monomorphism. Let $f^k:G\to G$ denote the $k$-th iterate of $f$. We call $f$ \textbf{strongly scale-invariant} if $f(G)$ is a proper finite index subgroup of $G$ and if $\cap_{k\geq0} f^k(G)$ is finite. We call $G$ \textbf{strongly scale-invariant} if it admits a strongly scale-invariant monomorphism.
\end{definition}

Using this stronger definition, it is very natural to ask a question similar to that of Benjamini, resulting in the following conjecture.

\begin{conjecture}[Benjamini-Nekrashevych-Pete]

Let $G$ be a finitely generated strongly scale-invariant group. Then $G$ is virtually nilpotent.

\label{Conj: main}
    
\end{conjecture}

By today, the conjecture is proved under extra assumptions: on the group, on the associated group chain and on the discriminant group of the associated group action. Van Limbeek showed in \cite{Limbeek} that the conjecture is true if we assume all $f^n(G)$ to be normal subgroups of $G$. Hurder, Lukina and van Limbeek proved the conjecture in the case when the discriminant group $D_f$, which is an invariant associated to the dynamical system induced by $f$, is finite \cite{hurder2021cantor}. Finally, Deré showed in \cite{Dere} that the conjecture is true for all virtually polycyclic groups.

The proof of Deré requires understanding of $\q$-algebraic hulls of polycyclic groups, and of the Malcev completion of nilpotent groups. Instead, in our paper we aim to give a very elementary proof using only linear algebra and some basic facts from abstract algebra for a smaller class of virtually polycyclic groups, namely that of semi-direct products of the form $G=\z^n\rtimes_A \z$, where we have $A\in \text{GL}_n(\z)$. 

The main idea is that $G$ is virtually nilpotent if and only if all eigenvalues of $A$ are roots of unity. Then the assumption that our group is not virtually nilpotent forces an injective group homomorphism $f:G\to G$ to be surjective in the second coordinate. This in turn gives us a non-zero fixed point of $f$. Since $G$ has no non-trivial finite order elements, the intersection $\cap_{n\geq0} f^n(G)$ is infinite. This prevents $f$ from being strongly scale-invariant, proving the following theorem.

\begin{theorem}
    Let $n$ be a positive integer, $A\in \text{GL}_n(\z)$, and $G=\z^n\rtimes_A \z$. Suppose that $G$ admits a strongly scale-invariant homomorphism $f:G\to G$. Then $G$ is virtually nilpotent. \label{The: main}
\end{theorem}

The rest of the paper is organized as follows: in Section \ref{Preliminaries} we prove a few technical results required for Theorem \ref{The: main}. In Section \ref{Main} we provide proofs of several statements, assuming that $G$ admits an injective endomorphism but is not virtually nilpotent, that together imply Theorem \ref{The: main}. 

\section{Preliminaries}

\label{Preliminaries}

A semi-direct product $G=\z^n\rtimes_A \z$ is given by a matrix $A\in \text{GL}_n(\z)$. Recall that the group law on such groups is given by the following product: $$(v,z)\star (w,c):=(v+A^zw,z+c).$$ In future use we will omit the notation $\star$ for convenience, and we will write $\z^n \rtimes \z$ if the matrix $A$ is either known from the context or if it is not important. Denote by $e_i\in \z^n$ for $i=1,...,n$ the $i$-th standard coordinate vector. Furthermore, with $I\in \text{GL}_n(\z)$ we will mean the identity matrix. 

Our goal is to give an elementary proof of Theorem \ref{The: main}. So, in particular, we must specify under which condition $G$ (or a subgroup of $G$) is nilpotent. We use the following definition.

\begin{definition}[\cite{Rotman}, Section 5.3]
    Let $H$ be a group. Iteratively define the groups $\gamma_k(H)$ by $\gamma_0(H)=H$ and for $k>0$: $$\gamma_k(H)=[\gamma_{k-1}(H),H].$$

    \begin{itemize}
        \item We call $H$ \textbf{nilpotent} if there exists some $k_0\in \z_{\geq 0}$ such that for all $k\geq k_0$ we have $\gamma_k(H)=\{0\}$.
        \item We call $H$ \textbf{virtually nilpotent} if $H$ contains a nilpotent finite index subgroup.
    \end{itemize}
\end{definition}

A direct computation shows that, for semi-direct products $G$ as above, the commutator subgroup $\gamma_1(G)$ is generated by elements of the form $$[(v,z),(w,c)]=((A^z-I)w-(A^c-I)v,0).$$ Notice that if $z=0$, the term $A^z-I$ vanishes. It follows that the groups $\gamma_k(G)$ are generated by all elements of the form \begin{align}
\Big((I-A^{c_{k-1}})\cdots (I-A^{c_1})((A^z-I)w-(A^c-I)v),0\Big)
\label{Eq: 1}
\end{align} for integers $z,c,c_1,...,c_{k-1}$ and any $v,w\in \z^n$. Next, it is also useful to know some elements that are contained in finite index subgroups of $G$.

\begin{lemma}
    Let $H\subseteq G=\z^n\rtimes_A \z$ be a finite index subgroup. Then for each $v\in \z^n$ there exists a non-zero integer $a$ such that $(av,0)\in H$. Similarly, there exists a non-zero integer $b$ such that $(0,b)\in H$.

    \label{Le: finite index, then contains (a,0)}
\end{lemma}

\begin{proof}

    Choose $v\in \z^n$. As $H$ is finite index in $G$, there must exist $a,a'\in \z$ with $a\neq a'$ such that we have an equality of cosets $(av,0)H=(a'v,0)H$. Multiplying on the left with $(-a'v,0)$ gives $$((a-a')v,0)H=(-a'v,0)(av,0)H=(-a'v,0)(a'v,0)H=H.$$ Hence $((a-a')v,0)H=H$, which holds if and only if $((a-a')v,0)\in H$. 
    
    The argument for the $(0,b)$ is similar, where instead we look at the second coordinate.
\end{proof}

Recall that a matrix $N$ is called nilpotent, if there exists a positive integer $k$ such that $N^k=0$. The following result is well-known; we include a proof here for completeness.

\begin{proposition}
    Let $A\in \text{GL}_n(\z)$ be a matrix and consider the semi-direct product $G=\z^n\rtimes_A \z$ induced by $A$. The following are equivalent:

    \begin{itemize}
        \item[(i)] The group $G$ is virtually nilpotent.
        \item[(ii)] All eigenvalues of $A$ are roots of unity.
        \item[(iii)] There exists an integer $z\neq 0$ such that $A^z-I$ is nilpotent.
    \end{itemize}

    \label{Prop: Virtually nilpotent iff nilpotent matrix iff eigenvalues roots of unity}
\end{proposition}

\begin{proof}
    \hyperref[Prop: Virtually nilpotent iff nilpotent matrix iff eigenvalues roots of unity]{(i) $\Rightarrow$ (ii)}: Suppose that $G$ contains a nilpotent finite index subgroup $H$. Then there exists some $k$ such that $\gamma_k(H)$ is trivial. Consider any $v\in \z^n$ and choose $a,b\in \z\setminus \{0\}$ such that $(av,0)\in H$ and $(0,b)\in H$, which is possible by Lemma \ref{Le: finite index, then contains (a,0)}.

    Let $v_0:=(av,0)\in H=\gamma_0(H)$, and inductively define $$v_i:=[v_{i-1},(0,b)]\in \gamma_i(H)=[\gamma_{i-1}(H),H].$$ By applying the formula for commutators as given in (\ref{Eq: 1}), we get $$v_k=(a (I-A^b)^kv,0)\in \gamma_k(H)=\{(0,0)\}.$$ Hence $a (I-A^b)^kv=0$, so as $a\neq 0$, we obtain $$(I-A^b)^kv=0,$$ which must hold for every $v\in \z^n$. Therefore $(I-A^b)^k$ must be the zero matrix. If $\lambda$ is any eigenvalue of $A$ with an eigenvector $w$, then we get $$0=(I-A^b)^kw=(1-\lambda^b)^kw.$$ We conclude that $1-\lambda^b=0$ holds, so that $\lambda$ must be a root of unity.

    \hyperref[Prop: Virtually nilpotent iff nilpotent matrix iff eigenvalues roots of unity]{(ii)$\Rightarrow$(iii)}: Let $\lambda_1,...,\lambda_k\in \c$ denote the eigenvalues of $A$, and let $r$ be some positive integer such that $1=\lambda_1^r=...=\lambda_k^r$. Let $J$ denote the (complex) Jordan normal form of $A$ and let $S$ be an invertible matrix with $A=SJS^{-1}$. As the eigenvalues lie on the diagonal of $J$, the matrix $J^r-I$ is a strictly upper triangular matrix. So in particular it is nilpotent. Now we can write $$A^r-I=SJ^rS^{-1}-I=S(J^r-I)S^{-1},$$ so that $A^r-I$ is also nilpotent.

    \hyperref[Prop: Virtually nilpotent iff nilpotent matrix iff eigenvalues roots of unity]{(iii)$\Rightarrow$(i)}: Now suppose there exists some non-zero integer $z$ such that $A^z-I$ is nilpotent. Say, $(A^z-I)^k=0$ for a positive integer $k$. Consider the finite index subgroup $H\leq G$ generated by all $(e_i,0)$ and $(0,z)$.

First of all, we notice that for any positive integer $r$ we have $$A^{rz}-I=(A^z-I)(A^{(r-1)z}+A^{(r-2)z}+...+I)$$ and $$A^{-rz}-I=-A^{-rz}(A^{rz}-I)$$ so that in particular for any choice of integers $r_1,...,r_k$ there exists some matrix $B\in \z[A]$ such that $$(A^{r_kz}-I)(A^{r_{k-1}z}-I)\cdots (A^{r_1z}-I)=(A^z-I)^kB = 0.$$ Here we can interchange the order of the expressions as they are all polynomials in the commutative ring $\z[A]$. The last coordinate of any element $(w,c)\in H$ will be a multiple of $z$, so using our computations of $\gamma_k(H)$ as in (\ref{Eq: 1}), we get that $\gamma_k(H)=0$. That is, $H$ is nilpotent.
\end{proof}

If $G=\z^n\rtimes_A \z$ is not virtually nilpotent, then by Proposition \ref{Prop: Virtually nilpotent iff nilpotent matrix iff eigenvalues roots of unity} and since $\det(A)\in \{\pm1\}$, the matrix $A$ must have at least two eigenvalues of absolute value not equal to 1. Thus we obtain the following.

\begin{lemma}
        Let $A\in \text{GL}_n(\z)$ be a matrix with at least one eigenvalue that is not a root of unity. Then for all non-zero integers $s$ the dimension of $\ker(A^s-I)$ is at most $n-2$. 

        \label{Le: dim ker A^s-I at most n-2}
    \end{lemma}

    Finally, nearing the end of the proof of Proposition \ref{Prop: properties of f}, we will use the following fact about invertible matrices over $\z$.

    \begin{lemma}
    Let $A\in \text{GL}_n(\z)$ and for each $m\in \z_{>0}$ define the matrix $$B_m:=\sum_{k=0}^{m-1} A^k.$$ Then, for each non-zero integer $r$, there exists some $m$ such that all entries of $B_m$ are divisible by $r$.

    \label{Le: sums of powers will eventually be divisible}
\end{lemma}

\begin{proof}

    Without loss of generality we may assume $r>1$. For this proof, let $R:=\mat_n(\z/r\z)$ and for any matrix $M\in \mat_n(\z)$ let $\widetilde{M}\in R$ denote the matrix with the same entries as $M$ viewed as elements of $\z/r\z$. Notice that reducing mod $r$ commutes with adding and multiplying matrices.

    As $\det(A)=\pm 1$, the determinant of $\widetilde{A}$ must be an invertible element of $\z/r\z$. Hence $\widetilde{A}$ is invertible over $R$. However, $R$ is a finite ring, so $\widetilde{A}$ must have finite (multiplicative) order. Say $\widetilde{A}^l=\widetilde{I}$ for $l\in \z_{>0}$. Now for $m=r\cdot l$ we get: $$\widetilde{B_{r\cdot l}}=\sum_{k=0}^{r\cdot l-1}\widetilde{A}^k=r\left(\widetilde{I}+\widetilde{A}+...+\widetilde{A}^{l-1}\right)=0.$$ But for any $m$, the matrix $\widetilde{B_m}$ is the reduction of the matrix $B_m$. So with $m=r\cdot l$ the matrix $B_m$ must only have entries that are all divisible by $r$.
    
\end{proof}

\section{Main result}

\label{Main}

Using the tools obtained in Section \ref{Preliminaries}, we can point out some properties of injective endomorphisms $f:G\to G$, where we assume that $G$ is not virtually nilpotent. Every statement in this proposition holds given all previous statements, and we use that $G$ is not virtually nilpotent only in the proof of (ii).

\begin{proposition}
    Fix some integer $n>0$ and let $A\in \text{GL}_n(\z)$. Let $G:=\z^n\rtimes_A\z$ and suppose $f:G\to G$ is an injective homomorphism. Let $g:G\to \z^n$ and $h:G\to \z$ be maps such that $f(v,z)=(g(v,z),h(v,z))$. If $G$ is not virtually nilpotent, then the following statements hold:

    \begin{itemize}
        \item[(i)] Consider the set $\z^{n+1}=\z^n\times \z$ with either the Abelian group structure, or the group structure obtained from the semi-direct product given by $A$. Then the set-wise map $h:\z^{n+1}\to \z$ becomes a group homomorphism.
        \item[(ii)] We have an inclusion $\z^n\times \{0\}\subseteq \ker h$.
        \item[(iii)] There exists an invertible matrix $F\in \text{Mat}_n(\z)$ such that for all $v\in \z^n$ we have $g(v,0)=Fv$.
        \item[(iv)] The homomorphism $h:G\to\z$ is surjective.
        \item[(v)] The composition $f^2:=f\circ f$ has a non-trivial fixed point.
    \end{itemize}

    \label{Prop: properties of f}
\end{proposition}

\begin{proof}

   Before we prove the statements, we consider some facts about $f$. Using that $f$ is a group homomorphism, we have $g(0,0)=0,h(0,0)=0$, and $$f(v,z)f(w,c)=(g(v,z)+A^{h(v,z)}g(w,c),h(v,z)+h(w,c))$$ equals $$f(v+A^zw,z+c)=(g(v+A^zw,z+c),h(v+A^zw,z+c)).$$

    This gives us two equations \begin{align}
        g((v,z)(w,c))=g(v+A^zw,z+c)&=g(v,z)+A^{h(v,z)}g(w,c)\label{Eq: 2}\\
        h((v,z)(w,c))=h(v+A^zw,z+c)&=h(v,z)+h(w,c) \label{Eq: 3}
    \end{align}

    \textit{Proof of} \hyperref[Prop: properties of f]{\textit{(i)}} Recall that the group operation of $G$ is given by $$(v,z)(w,c)=(v+A^zw,z+c).$$ In particular this product adds the last two coordinates together, so that the projection map $G\to \z$ on the last coordinate is a group homomorphism. Post-composing $f$ with this projection gives us $h$, showing that $h:G\to \z$ is a group homomorphism. 

    Choosing $w=0$ and $z=0$ in (\ref{Eq: 3}) gives $h(v,z)=h(v,0)+h(0,c)$. So since $(v,0)(w,0)=(v+w,0)$ and $(0,z)(0,c)=(0,z+c)$, we obtain $$h(v+w,z+c)=h(v,0)+h(w,0)+h(0,z)+h(0,c)=h(v,z)+h(w,c).$$ That is, $h:\z^{n+1}\to \z$ is also a group homomorphism.

    \textit{Proof of} \hyperref[Prop: properties of f]{\textit{(ii)}} Consider the following two sets $$V:=\{v\in \z^n\mid h(v,0)=0\}$$ $$X:=\{g(v,0)\in \z^n\mid v\in V\}$$ If we view $\z^n$ as the subgroup $\z^n\times\{0\}\subset G$, then $V$ is the kernel of the induced group homomorphism $h|_{\z^n}:\z^n\times \{0\}\to \z$. Hence $V$ is a group itself. But then for $v,w\in V$ we have \begin{align}g(v+w,0)=g(v,0)+A^{h(v,0)}g(w,0)=g(v,0)+g(w,0)\label{Eq: 4}\end{align} so that $X$ is also a group.

    As they are subgroups of $\z^n$, we have induced $\q$-vector spaces $V_\q:=V\otimes \q$ and $X_\q:= X\otimes \q$, where $V_\q$ is the kernel of the $\q$-linear map $h\otimes \q:\q^n\to \q$. In particular $V_\q$ has dimension at least $n-1$. 

    Next, let $v\in V$ and $w\in \z^n$. Then we have two ways of writing $g(v+w,0)$. Namely: $$g(v,0)+g(w,0)=g(v+w,0)=g(w+v,0)=g(w,0)+A^{h(w,0)}g(v,0).$$ Therefore $g(v,0)=A^{h(w,0)}g(v,0)$, and in turn this implies \begin{align}
        (I-A^{h(w,0)})g(v,0)=0. \label{Eq: 5}
    \end{align} Notice that this implies that $X\subseteq \ker(I-A^{h(w,0)})$. If we embed $X_\q$ in $\q^n$ by taking $\q$-linear combinations, then it also implies $X_\q\subseteq \ker(I-A^{h(w,0)})$. 
    
    Assume that $h(w,0)\neq 0$. Then by Lemma \ref{Le: dim ker A^s-I at most n-2}, we see that $X_\q$ has dimension at most $n-2$. By (\ref{Eq: 4}) we have a $\q$-linear map $g\otimes \q:V_\q\to X_\q$, so by comparing dimensions the kernel must have dimension at least 1. Multiplying by some integer shows that there exists some non-zero $v\in V$ such that $g(v,0)=0$. Hence we get $$f(v,0)=(g(v,0),h(v,0))=(0,0),$$ contradicting that $f$ is injective. Hence we are not allowed to use Lemma \ref{Le: dim ker A^s-I at most n-2}, so we must have $h(w,0)=0$. As a result, $h(\z^n\times \{0\})=\{0\}$.

    \textit{Proof of} \hyperref[Prop: properties of f]{\textit{(iii)}} Firstly, by (ii) for any $v,w\in \z^n$ we have (\ref{Eq: 4}). So we must have some matrix $F\in \mat_n(\z)$ such that for all $v\in \z^n$ we have $g(v,0)=Fv$. If $F$ would not be invertible, then there exists some non-zero $v\in \z^n$ such that $Fv=0$. Hence we get $$f(v,0)=(Fv,h(v,0))=(0,0),$$ which contradicts that $f$ is injective.

    \textit{Proof of} \hyperref[Prop: properties of f]{\textit{(iv)}} Let $v\in \z^n$. Then we have: $$FAv=g(Av,0)=g((0,1)(v,0)(0,-1))=g(0,1)+A^{h(0,1)}g(v,0)+A^{h(0,1)}g(0,-1).$$ But observe that $$0=g(0,1-1)=g(0,1)+A^{h(0,1)}g(0,-1).$$ Therefore we obtain the equation $$FAv=A^{h(0,1)}g(v,0)=A^{h(0,1)}Fv.$$ Since this must hold for all $v\in \z^n$, we have $$FAF^{-1}=A^{h(0,1)},$$ so that $A$ and $A^{h(0,1)}$ are similar matrices. In particular their eigenvalues must coincide, so since by Proposition \ref{Prop: Virtually nilpotent iff nilpotent matrix iff eigenvalues roots of unity} we may assume that $A$ has eigenvalues that are not roots of unity, we must have $h(0,1)\in \{\pm 1\}$. Hence $h$ is surjective.

    \textit{Proof of} \hyperref[Prop: properties of f]{\textit{(v)}} By (ii) and (iv) we have $h(0,1)\in \{\pm 1\}$. We can take $f^2=f\circ f$ instead of $f$, so that $h(0,1)=1$. By (i) and (ii) we obtain the equation $h(v,z)=z$. We consider two cases: when $F-I$ is invertible, and when $F-I$ is not invertible.

    If it is not invertible, then there exists some non-zero $v\in \z^n$ such that $Fv=v$. In particular we have $$f(v,0)=(g(v,0),h(v,0))=(Fv,0)=(v,0),$$ so that $(v,0)$ is a fixed point of $f$. 

    If $F-I$ is invertible, then choose some positive integer $z$ such that $$(I-F)^{-1}(I+A+...+A^{z-1})$$ is an integer matrix. This is possible by Lemma \ref{Le: sums of powers will eventually be divisible} by choosing $r=\det(I-F)$. Now consider the pair $(v,z)$ with $$v:=(I-F)^{-1}(I+A+A^2+...+A^{z-1})g(0,1)\in \z^n.$$ Using that $$F(I-F)^{-1}=-(I-F)(I-F)^{-1}+(I-F)^{-1}=-I+(I-F)^{-1}$$ and $$g(0,z)=g(0,1)+Ag(0,1)+A^2g(0,1)+...+A^{z-1}g(0,1)=(I+A+...+A^{z-1})g(0,1),$$ where we have used (\ref{Eq: 2}) iteratively with $g(0,z)=g(0,1+(z-1))$, we get $$f(v,z)=(Fv+g(0,z),z)=(-g(0,z)+v+g(0,z),z)=(v,z).$$ Therefore, we have found a non-zero fixed point $(v,z)$ of $f$.
    
\end{proof}

As a consequence, we can prove the main theorem.

\begin{proof}[Proof of Theorem \ref{The: main}]
    Consider some semi-direct product $G=\z^n\rtimes \z$ and let $f:G\to G$ be a strongly scale-invariant homomorphism. 
    
    If $G$ is virtually nilpotent, we are done, so assume that $G$ is not virtually nilpotent. By Proposition \ref{Prop: properties of f} we get that $f^2$ has some non-zero fixed point $(v,z)$. Since for all positive integers $k$ we have $f^{k}(G)\subseteq f^{k-1}(G)$, we get $$(v,z)\in \bigcap_{k\geq 0} f^k(G).$$ However, $(v,z)$ has infinite order, as when $z\neq 0$, then $(v,z)^m$ has second coordinate equal to $z\cdot m$, and if $z=0$, then the first coordinate of $(v,z)^m$ is $v\cdot m$. We conclude that this intersection cannot be finite, which contradicts the assumption of $f$. That is, $G$ must be virtually nilpotent.
\end{proof}

\begin{acknowledgement}
    I would like to thank my supervisor Olga Lukina for all the continued help and support in writing my first paper. I am grateful for the help of Wim Nijgh for giving inspiration for the proof of Lemma \ref{Le: sums of powers will eventually be divisible}.
\end{acknowledgement}

Mathematical Institute \\ Leiden University \\ 
P.O. Box 9512 \\
2300 RA Leiden \\
The Netherlands\\\href{mailto:d.a.wardell@math.leidenuniv.nl}{d.a.wardell@math.leidenuniv.nl}

\end{document}